\documentclass[11pt, letterpaper]{amsart}
\usepackage{graphicx, amssymb, hyperref}

\newtheorem{thm}{Theorem}[section]

\newtheorem{lem}[thm]{Lemma}
\newtheorem{prop}[thm]{Proposition}
\theoremstyle{definition}
\newtheorem{defn}[thm]{Definition}

\theoremstyle{remark}
\newtheorem{rem}[thm]{Remark}
\newtheorem{exa}[thm]{Example}
\numberwithin{equation}{section}

\newcommand{\set}[1]{\left\{#1\right\}}
\newcommand{\Real}{\mathbb R}

\newcommand{\Natural}{\mathbb N}
\newcommand{\nin}{n \in \Natural}
\newcommand{\kin}{k \in \Natural}

\newcommand{\such}{{\ | \ }}
\newcommand{\limn}{\lim_{n \to \infty}}

\newcommand{\dfn}{\, := \,}

\newcommand{\prob}{\mathbb{P}}

\newcommand{\conv}{\mathsf{conv}}

\newcommand{\qprob}{\mathbb{Q}}
\newcommand{\expec}{\mathbb{E}}
\newcommand{\expecp}{\expec_\prob}
\newcommand{\expecq}{\expec_\qprob}

\newcommand{\expecqg}{\expec_{\qprob_\gamma}}

\newcommand{\Lb}{\mathbb{L}}
\newcommand{\lz}{\Lb^0}

\newcommand{\lzp}{\lz_{+}}
\newcommand{\lzpp}{\lz_{++}}

\newcommand{\Ss}{\mathcal{S}}

\newcommand{\F}{\mathcal{F}}

\newcommand{\ud}{\mathrm d}

\newcommand{\inner}[2]{\left \langle #1 , \, #2 \right \rangle}

\newcommand{\num}{num\'eraire}

\newcommand{\U}{\mathbb{U}}

\newcommand{\pare}[1]{\left(#1\right)}
\newcommand{\bra}[1]{\left[#1\right]}
\newcommand{\dbra}[1]{[\kern-0.15em[ #1 ]\kern-0.15em]}
\newcommand{\dbraco}[1]{[\kern-0.15em[ #1 [\kern-0.15em[}
\newcommand{\dbraoc}[1]{]\kern-0.15em] #1 ]\kern-0.15em]}
\newcommand{\C}{\mathcal{C}}
\newcommand{\Cmax}{\C^{\mathsf{max}}}
\newcommand{\K}{\mathcal{K}}

\newcommand{\Tl}{\mathcal{T}}

\newcommand{\Cnum}{\C^{\mathsf{num}}}

\newcommand{\indic}{\mathbb{I}}

\newcommand{\absco}{{<\kern-0.53em<}}

\newcommand{\oC}{\overline{\C}}
\newcommand{\oCmax}{{\oC^{\mathsf{max}}}}

\begin{document}

\title[Maximality and num\'eraires in convex subsets of $\lzp$]{Maximality and num\'eraires in convex sets of nonnegative random variables}%

\author{Constantinos Kardaras}%
\address{Constantinos Kardaras, Statistics Department, London School of Economics and Political Science}%
\email{k.kardaras@lse.ac.uk}%


\date{\today}%
\begin{abstract}
We introduce the concepts of max-closedness and \num s of convex subsets of $\lzp$, the nonnegative orthant of the topological vector space $\lz$ of all random variables built over a probability space, equipped with a topology consistent with  convergence in probability. Max-closedness asks that maximal elements of the closure of a set already lie on the set. We discuss how \num s  arise naturally as strictly positive optimisers of certain concave monotone maximisation problems. It is further shown that the set of \num s of a convex, max-closed and bounded set of $\lzp$ that contains at least one strictly positive element is dense in the set of its maximal elements.
\end{abstract}

\maketitle


\section*{Introduction}

\subsection*{Discussion}

Let $\lz$ denote the set of all (equivalence classes of real-valued) random variables built over a probability space, equipped with a metric topology under which convergence of sequences coincides with convergence in probability. Denote by $\lzp$ the nonnegative orthant of $\lz$. In many problems of interest---notably, in the field of mathematical finance---one seeks maximisers of a concave and strictly monotone (increasing) functional $\U$ over convex set $\C \subseteq \lzp$. In order to ensure that such optimisers exist, some closedness property of $\C$ should be present. The strict monotonicity of $\U$ \emph{a priori} implies that, if optimisers exist, they must be maximal elements of $\C$ with respect to the natural lattice structure of $\lz$; therefore, a natural condition to enforce is that maximal points of the closure of $\C$ already lie in $\C$. We then refer to the set $\C$ as being \textsl{max-closed}, and the collection $\Cmax$ of all its maximal elements is regarded as the ``outer boundary'' of $\C$.

Concave maximisation problems as the one described above are particularly amenable to first-order analysis. Morally speaking, a maximiser of a concave functional $\U$ over $\C$ should also be a maximiser of a \emph{nice} nonzero linear functional over $\C$. When \emph{nice} means \emph{continuous}, such an element is called a \textsl{support point} of $\C$ in traditional functional-analytic framework, and existence of a supporting nonzero continuous linear functional is typically provided by an application of the geometric form of the Hahn-Banach theorem. Unfortunately, $\lz$ is rather unsuitable\footnote{Note, however, that whenever $\C \subseteq \lzp$ is a convex and bounded (in measure) set, there exists a probability $\qprob$, equivalent to the underlying one, such that $\C$ is bounded in $\Lb^1(\qprob)$---see discussion after Theorem \ref{thm: bishop-phelps}; although this sometimes facilitates the analysis on $\C$, in general the $\Lb^0$-topology does not coincide with the $\Lb^1(\qprob)$-topology on $\C$.} for application of standard convex-analytic techniques. More  precisely, when the probability space is non-atomic:
\begin{itemize}
	\item $\lz$ fails to be locally convex, which implies that a rich body of results (including the Hahn-Banach theorem) cannot be used;
	\item the topological dual of $\lz$ contains only the zero functional \cite[Theorem 2.2, page 18]{MR808777}; in particular, as the Namioka-Klee theorem \cite{NK} suggests, there is no real-valued nonzero positive linear functional on $\lz$.
\end{itemize}
In particular, convex sets in $\lz$ \emph{a fortiori} lack support points according to the usual definition. The previous issue notwithstanding, this work aims at exploring special elements of convex subsets of $\lzp$ which can be regarded as support points. More precisely, we discuss the notion of a \textsl{\num} $g$ of a set $\C \subseteq \lzp$, asking that $g$ is strictly positive (in the sense that $\set{g = 0}$ is a null set) and there exists a probability measure $\qprob$, equivalent to the underlying probability measure, such that $\expecq \bra{f/g} \leq 1$ holds for all $f \in \C$, where ``$\expecq$'' denotes expectation under $\qprob$. As is argued in the article (see Remark \ref{rem: supp_points}), \num s are closely related to support points in the classical sense, where the supporting ``dual element'' corresponds to a $\sigma$-additive, $\sigma$-finite, positive measure, equivalent to the underlying probability measure. Furthermore, by means of the rather wide-encompassing example in \S~\ref{subsec: util_max_num} it is rigorously illustrated that optimisers for a large class of concave monotone maximisation problems over convex sets are indeed \num s according to the previous definition.

N\'umeraires are maximal elements of convex sets. A natural question is to explore the richness of $\Cnum$, the class of \num s of a set $\C \subset \lzp$, in its outer boundary $\Cmax$. To ensure that the discussion is not void, it is established that if $\C$ contains a strictly positive element and is convex, max-closed and bounded in $\lz$, then $\Cnum \neq \emptyset$. It is further shown by an example in \S~\ref{subsec: exa} that there exists a space $\lz$ and a convex and compact set $\C \subset \lzp$ containing an element in $\Cmax \setminus \Cnum$. (An infinite-dimensional space is required for such example. In finite-dimensional Euclidean spaces all boundary points of a closed and convex set are support points, and it can be shown in the present non-standard set-up that any strictly positive maximal element is a \num. Note also that in infinite-dimensional spaces there are examples of proper closed convex subsets that have \emph{no} support points---see \cite{MR0149241}.) On the positive side, it is shown in Theorem \ref{thm: bishop-phelps} for convex, max-closed and bounded sets $\C \subset \lzp$ that contain at least one strictly positive element, $\Cnum$ is dense in $\Cmax$. In the context of Banach spaces, Bishop and Phelps Theorem \cite[Theorem 7.43]{MR2378491} states that support points of closed and convex sets are dense on the boundary of the set. Therefore, Theorem \ref{thm: bishop-phelps} can be seen as an analogue of the Bishop-Phelps theorem in an extremely non-standard environment, where the topological space in question fails to even be locally convex.

\smallskip

The structure of the paper is simple. Section \ref{sec: max} introduces and discusses maximal elements and max-closedness, Section \ref{sec: num} introduces \num s and shows that there exist maximal elements that are not \num s, while in Section \ref{sec: BP} the aforementioned density of \num s in the maximal elements for convex, max-closed and bounded in $\lz$ subsets of $\lzp$ that contains a strictly positive element is stated and proved.

\subsection*{Preliminaries}

Throughout the paper, $\lz$ denotes the set of all real-valued random variables over a probability space $(\Omega, \F, \prob)$. The usual practice of not distinguishing a random variable from the equivalence class (modulo $\prob$) it generates is followed. All relationships between elements of $\lz$ are to be understood in the $\prob$-a.s. sense. Define $\lzp \dfn \set{f \in \lz \such f \geq 0}$ to be the nonnegative orthant of $\lz$; furthermore, let $\lzpp$ be the class of all $f \in \lzp$ such that $f > 0$.

We use $\qprob \sim \prob$ (respectively, $\qprob \ll \prob$) to denote that $\qprob$ is a probability measure that is equivalent with  (respectively, absolutely continuous with respect to) $\prob$. The symbol $\expecq$ is used to denote expectation with respect to $\qprob \ll \prob$; we simply use $\expec$ instead of $\expecp$ for expectation under $\prob$.

The topology on the vector space $\lz$ is the one induced by the translation-invariant metric $\lz \times \lz \ni (f, g) \mapsto \expec \bra{1 \wedge |f-g|}$, where ``$\wedge$'' is used to denote the minimum operation. With the above definition, $\lz$ becomes a complete metric space and $\lzp$ a closed and convex subset. Convergence of sequences under this topology is convergence in $\prob$-measure. (In fact, the topology only depends on the equivalence class of $\prob$.) Unless explicitly stated otherwise, any topological property (closedness, etc.) pertaining to subsets of $\lz$ will be understood under the aforementioned topology.

For $\C \subseteq \lzp$, $\oC \subseteq \lzp$ will denote the closure of $\C$. A set $\C \subseteq \lzp$ will be called \textsl{bounded} if $\lim_{\ell \to \infty} \sup_{f \in \C} \prob[f > \ell] = 0$---as can be easily seen, the last property coincides with boundedness of $\C$ when $\lz$ is viewed as a topological vector space \cite[Definition 5.36]{MR2378491}. If $\C \subseteq \lzp$ is bounded, it is straightforward to check that $\oC$ is bounded as well. 
Finally, $\Ss \subseteq \lzp$ will be called \textsl{solid} if the conditions $g \in \Ss$, $f \in \lzp$ and $f \leq g$ imply $f \in \Ss$.

\section{Maximal Elements and Max-Closedness}
\label{sec: max}

An element $f \in \C \subseteq \lzp$ is called \textsl{maximal in $\C$} if the conditions $f \leq g$ and $g \in \C$ imply $f = g$; the notation $\Cmax$ is used to denote the set of all maximal elements in $\C$.

The next definition introduces a concept of closedness that additionally takes into account the lattice structure of $\lz$. It is exactly tailored for problems related to concave monotone maximisation, as is shown in Proposition \ref{prop: conv_mon_max} below.

\begin{defn} \label{dfn: max-closed}
A set $\C \subseteq \lzp$ will be called \textsl{max-closed} if $\oCmax \subseteq \C$.
\end{defn}

In words, max-closedness asks that all maximal elements in the closure of a set are already contained in the set itself. Max-closedness is a weaker property than closedness (see Example \ref{exa: bad_C} later on), and has played an important background role in the proof of the Fundamental Theorem of Asset Pricing in \cite{MR1304434}.


The next result implies in particular that $\Cmax \neq \emptyset$ whenever $\C \neq \emptyset$ is max-closed and bounded. We omit the simple argument for its proof, which relies on a use of Zorn's lemma and has already appeared in \cite[proof of Lemma 4.3]{MR1304434} and \cite[paragraph after Theorem 3.1]{MR1647282}.

\begin{lem} \label{lem: zorn}
Let $\C \subseteq \lzp$ be max-closed and bounded. Then, for every $f \in \C$ there exists $h \in \Cmax$ with $f \leq h$.
\end{lem}

%

The next result gives an alternative definition of max-closedness for bounded subsets of $\lzp$.

\begin{lem} \label{lem: bdd_max_closed}
Suppose that $\C \subseteq \lzp$ is bounded. Then, $\C$ is max-closed if and only if $\oCmax = \Cmax$.
\end{lem}

\begin{proof}
If $\oCmax = \Cmax$, the fact that $\oCmax \subseteq \C$ immediately implies that $\C$ is max-closed.

Suppose now that $\C$ is max-closed. Since $\C \subseteq \oC$, the set-inclusion $\oCmax \subseteq \C$ is equivalent to $\oCmax \subseteq \Cmax$. Now, let $f \in \Cmax$. Since $\oC$ is closed (in particular, max-closed) and bounded (because $\C$ is bounded; note that boundedness of $\C$ is only required in this is the part of the proof), Lemma \ref{lem: zorn} implies the existence of $g \in \oCmax$ with $f \leq g$. Since $\C$ is max-closed, it follows that $g \in \C$, which with $f \in \Cmax$ implies that $f= g$. Therefore, $f \in \oCmax$, which shows that  $\Cmax \subseteq \oCmax$ and completes the proof.
\end{proof}

\begin{rem}
If $\C \subseteq \lzp$ fails to be bounded, Lemma \ref{lem: bdd_max_closed} is not necessarily true. For example, take $\Omega = (0, 1)$, $\F$ be the Borel $\sigma$-field on $\Omega$, and let $\prob$ be Lebesgue measure on $(\Omega, \F)$. Define $\C = \set{1} \cup \set{f \in \lzp \such \prob \bra{f < 1} > 0}$. Then, $\oC = \lzp$, so that $\oCmax = \emptyset$ and $\C$ is trivially max-closed. However, $1 \in \Cmax$, which shows that $\Cmax = \oCmax$ is violated.
\end{rem}

The following example demonstrates, \emph{inter alia}, that max-closedness is a strictly weaker notion than closedness.

\begin{exa} \label{exa: bad_C}
Let $\Omega = (0, 1)$, $\F$ be the Borel $\sigma$-field on $\Omega$, and let $\prob$ be Lebesgue measure on $(\Omega, \F)$. Consider $\C = \set{f \in \lzp \such \expec \bra{f} = 1}$. The set-inclusion $\oC \subseteq \set{f \in \lzp \such \expec [f] \leq 1}$ follows from Fatou's lemma. Now, let $z_n \dfn n^{-1} \indic_{(0, n^{-1})}$, so that $z_n \in \C$ for all $\nin$. Note that $\limn z_n = 0$. For any $f \in \lzp$ with $\expec [f] \leq 1$, the $\C$-valued sequence $\pare{f + (1 - \expec [f]) z_n}_{\nin}$ converges to $f$, which shows that $f \in \oC$. It follows that $\oC = \set{f \in \lzp \such \expec [f] \leq 1}$; in particular, $\C$ is not closed. However, note that $\oCmax = \set{f \in \lzp \such \expec [f] = 1} = \Cmax$, which implies that $\C$ is max-closed.

In this setting, note that $\Ss = \set{f \in \lzp \such f \leq g \text{ for some } g \in \C}$ (the solid hull of $\C$) is equal to $\set{f \in \lzp \such \expec [f] \leq 1}$, which is closed. This did not happen by chance: it is shown in Proposition \ref{prop: max-closed-solid} that the solid hull of a convex, max-closed and bounded set is always closed.

Let us make one more observation. With $\partial \C$ denoting the topological boundary of $\C$, it actually holds that $\partial \C = \oC$. Indeed, for $f \in \oC = \set{f \in \lzp \such \expec [f] \leq 1}$ note that $(f + 2 z_n)_{\nin}$ is a $(\lzp \setminus \oC)$-valued sequence which converges to $f$.
\end{exa}

In the example above, $\oC$ turned out to be a much larger set than $\C$. Even though $\C$ is not closed, in many cases of interest the ``important'' elements of $\C$ lie on the ``outer boundary'' $\Cmax$ of $\C$. (As was seen in Example \ref{exa: bad_C}, the topological boundary of $\C \subseteq \lzp$ might be simply too large to provide useful information about optimal elements.) To this effect, the next result demonstrates that the notion of max-closedness ties nicely together with concave monotone maximisation. 

\begin{prop} \label{prop: conv_mon_max}
Suppose that $\U: \lzp \mapsto [- \infty, \infty)$ is concave, upper semi-continuous and monotone, the latter meaning that $\U(f) \leq \U(g)$ holds whenever $f \leq g$. Let $\C \subset \lzp$ be convex, max-closed and bounded. Then, there exists $g \in \Cmax$ such that $\sup_{f \in \C} \U(f) = \U(g) < \infty$.
\end{prop}

\begin{proof}
Note that $\oC$ is convex, closed and bounded. Since $\U$ is concave and upper semi-continuous, \cite[Lemma 4.3]{MR2651515} implies the existence of $g_0 \in \oC$ such that $\U(g_0) = \sup_{f \in \oC} \U(f)$. Furthermore, since $\oC$ is closed (in particular, max-closed) and bounded, Lemma \ref{lem: zorn} implies that there exists $g \in \oCmax$ such that $g_0 \leq g$. Since $\U$ is monotone, $\U(g_0) \leq \U(g)$ holds, which means that $\U(g) = \sup_{f \in \oC} \U(f)$. Finally, since $\C$ is max-closed, $g \in \Cmax$ follows.
\end{proof}

\begin{rem}
Functions $\U : \lzp \mapsto [-\infty, \infty)$ with the properties in the statement of Proposition \ref{prop: conv_mon_max} appear in problems of mathematical finance, where $\U$ represents a utility functional. An interesting---in terms of structure---example is given in \S~\ref{subsec: util_max_num}.
\end{rem}

The next result (which was announced in Example \ref{exa: bad_C}) associates convexity, max-closeness, boundedness, solidity, and closedness. Before we state it, a definition is required. Let $(f_n)_{\nin}$ be a sequence in $\lzp$. Any sequence $(g_n)_{\nin}$ with the property that $g_n$ lies in the convex hull of $\set{f_n, f_{n+1}, \ldots}$ for all $\nin$ will be called a \textsl{sequence of forward convex combinations of $(f_n)_{\nin}$}.

\begin{prop} \label{prop: max-closed-solid}
Let $\C \subseteq \lzp$ be convex, max-closed and bounded, and define its solid hull $\Ss = \set{f \in \lzp \such f \leq g \text{ for some } g \in \C}$. Then, $\Ss$ is solid, convex, closed and bounded.
\end{prop}

\begin{proof}
It is straightforward to check that $\Ss$ is solid, convex and bounded. It remains to show that it is closed. Let $(f_n)_{\nin}$ be a $\Ss$-valued sequence converging to $f \in \lzp$; we shall establish that $f \in \Ss$. By passing to a subsequence if necessary, assume that $(f_n)_{\nin}$ converges $\prob$-a.s. to $f$. (The importance of $\prob$-a.s. convergence is that any sequence of forward convex combinations of $(f_n)_{\nin}$ also converges to $f$, which will be tacitly used in the proof later  on---$\Real$ is a locally convex space, while $\lz$ is not.) For each $\nin$, there exists $g_n \in \C$ such that $f_n \leq g_n$. Note that $\oC$ is convex, closed and bounded; then, \cite[Lemma A1.1]{MR1304434} implies the existence of a sequence of forward convex combinations of $(g_n)_{\nin}$ that $\prob$-a.s. converges to some $g \in \oC$. Since $(f_n)_{\nin}$ converges $\prob$-a.s. to $f$ and $f_n \leq g_n$ for all $\nin$, it follows that $f \leq g$. Now, invoking Lemma \ref{lem: zorn}, it follows that there exists $h \in \oCmax$ such that $g \leq h$. As $\oCmax \subseteq \C$ and $f \leq g$, we obtain that $f \leq h \in \C$, which implies that $f \in \Ss$.
\end{proof}

The final result of this section---Proposition \ref{prop: stab_conv_outer_bdry}---is concerned with ``stability'' of convergence of sequences to points of the outer boundary of convex subsets of $\lzp$. Before stating it, we mention the following result, which is a special case of \cite[Theorem 1.3]{MR3003684} and will be used thrice in the proof of Proposition \ref{prop: stab_conv_outer_bdry}.

\begin{thm} \label{thm: kz} Let $(f_n)_{n \in \Natural}$ be a sequence in $\lzp$ such that $\conv \pare{\set{f_n \such \nin}}$ is bounded. Assume that
$\limn f_n = g$ holds for some $g \in \lzp$. Then, the following statements are equivalent:
  \begin{enumerate}
  \item Every sequence of forward convex combinations of $(f_n)_{n \in
      \Natural}$ converges to $g$.
  \item If a sequence of forward convex combinations of
    $(f_n)_{n \in \Natural}$ is convergent, its limit is
     $g$.
  \end{enumerate}
(In the case $f = 0$, the equivalence of (1) and (2) holds even without assuming $\limn f_n = 0$.)

If any of the equivalent conditions above fail, the set $\K \subseteq \lzp$ of all possible limits of forward convex combinations of $(f_n)_{\nin}$ is such that $\set{g} \varsubsetneq \K$, and $g \leq h$ holds for all $h \in \C$.

\end{thm}

\begin{prop} \label{prop: stab_conv_outer_bdry}
Let $\C\subseteq \lzp$ be convex. Let $(f_n)_{\nin}$ be a $\C$-valued sequence converging to $g \in \Cmax$. Then,
\begin{enumerate}
	\item Any sequence of forward convex combinations of $(f_n)_{\nin}$ also converges to $g$.
	\item Any $\C$-valued sequence $(g_n)_{\nin}$ such that $f_n \leq g_n$ holds for all $\nin$ also converges to $g$.
\end{enumerate}
\end{prop}

\begin{proof}
In the sequel, $(f_n)_{\nin}$ is a $\C$-valued sequence that converges to $g \in \Cmax$.

Suppose that $(g_n)_{\nin}$ is a sequence of forward convex combinations of $(f_n)_{\nin}$ that converges to $h \neq g$. By Theorem \ref{thm: kz}, this would contradict the fact that $g \in \Cmax$. Therefore, \emph{any} convergent sequence of forward convex combinations of $(f_n)_{\nin}$ must have the same limit $g$ that $(f_n)_{\nin}$ has. Again, by Theorem \ref{thm: kz} it follows that \emph{all} sequences of forward convex combinations of $(f_n)_{\nin}$ converge to $g$, which establishes statement (1).

Now, pick any $\C$-valued sequence $(g_n)_{\nin}$ such that $f_n \leq g_n$ holds for all $\nin$, and define $\zeta_n \dfn g_n - f_n$ for $\nin$; then, $\zeta_n \in \lzp$ for all $\nin$. If $\limn \zeta_n = 0$ is established, $\limn f_n = g$ will imply $\limn g_n = g$. For $\nin$, let $\Tl_n$ denote the closure of the convex hull of $\set{\zeta_k \such k = n, n+1, \ldots}$, and set $\Tl_\infty \dfn \bigcap_{\nin} \Tl_n$. For $\psi \in \Tl_\infty$, there exists a sequence $(\psi_n)_{\nin}$ of forward convex combinations of $(\zeta_n)_{\nin}$ such that $\limn \psi_n = \psi$. Since $g \in \Cmax$, $\Tl_\infty$ cannot contain any $\psi \in \lzp$ with $\prob[\psi > 0] > 0$; indeed, if this was the case, using statement (1) of Proposition \ref{prop: stab_conv_outer_bdry} that was just proved, one would be able to construct a $\C$-valued sequence $(h_n)_{\nin}$ with $\limn h_n = g + \psi$, which would mean that $(g + \psi) \in \C$ and would contradict $g \in \Cmax$. On the other hand, as each $\Tl_n$, $\nin$, is convex, closed and bounded and $(\Tl_n)_{\nin}$ is a non-increasing sequence, it follows from \cite{MR2651515} that $\Tl_\infty \neq \emptyset$. We conclude that $\Tl_\infty = \set{0}$---in other words, all convergent sequences of forwards convex combinations of $(\zeta_n)_{\nin}$ converge to zero. Then, another application of Theorem \ref{thm: kz} (for the special case of zero limit) implies that $\limn \zeta_n = 0$, completing the proof of statement (2).
\end{proof}

\section{Num\'eraires} \label{sec: num}

\subsection{The \num \ property}

In the theory of financial economics, a convex set $\C \subseteq \lzp$ frequently models the class of all possible choices available for future consumption given (normalised) unit budget. Any element $g \in \C \cap \lzpp$ (note that $g$ is \emph{strictly} positive) can be used as a \num, in the sense of a benchmark under which the value of all other consumption choices is compared to; more precisely, for $f \in \C$, the random variable $f / g$ measures $f$ in units of $g$. We regard $g \in \C$ to be a ``good'' \num \ if there exists a valuation probability $\qprob \sim \prob$ that gives value at most one to all elements $f \in \C$ denominated in units of $g$. (For more motivation and discussion on the previous theme, we send the interested reader to \cite{MR1381678}.) 
 
\begin{defn} \label{dfn: num}
Let $\C \subseteq \lzp$ be such that $\C \cap \lzpp \neq \emptyset$. An element $g \in \C$ will be called a \textsl{\num \  of $\C$} if $g \in \lzpp$ and there exists $\qprob \sim \prob$ such that $\expecq \bra{f / g} \leq $ holds for all $f \in \C$. The set of all \num s of $\C$ is denoted by $\Cnum$.
\end{defn}

\begin{rem} \label{rem: num_is_max}
For $\C \subseteq \lzp$ with $\C \cap \lzpp \neq \emptyset$, it is straightforward to check that a \num \ of $\C \subseteq \lzp$ is a maximal element of $\C$; in other words, $\Cnum \subseteq \Cmax \cap \lzpp$.
\end{rem}

\begin{rem} \label{rem: supp_points}
One may offer a functional-analytic interpretation of \num s, in terms of ``support points'' of convex sets, as we now explain. For a measure (note that all measures will be assumed countably additive, non-negative and $\sigma$-finite) $\mu \sim \prob$, consider the linear mapping
\begin{equation} \label{eq: inner_prod}
\lzp \ni f \mapsto \inner{\mu}{f} \dfn \int_\Omega f \ud \mu \in [0, \infty].
\end{equation}
Let $\C \subseteq \lzp$ be such that $\C \cap \lzpp \neq \emptyset$. It is then straightforward to check that $g \in \C \cap \lzpp$ is a \num \  of $\C$ if and only if there exists there exists a measure $\mu \sim \prob$ such that $\sup_{f \in \C} \inner{\mu}{f} = \inner{\mu}{g} < \infty$. 
Although the mapping of \eqref{eq: inner_prod} fails to be continuous in general (in view of Fatou's lemma, it is at least lower semi-continuous), we may still regard a \num \ as a non-standard support point of $\C$. Note, however, that there are special properties involved in the definition of a \num \ $g$ of $\C$; not only does $g$ have to be a \emph{strictly} positive element, but also the ``supporting functional'' given by $\mu$ has to be strictly positive (in the sense that $\mu \sim \prob$) as well.
\end{rem}

\subsection{A canonical example} \label{subsec: util_max_num}

According to Proposition \ref{prop: conv_mon_max}, optimisers of concave, upper semi-continuous and monotone functionals over convex, max-closed and bounded sets $\C \subset \lzp$ exist and lie on $\Cmax$. As mentioned in the introductory discussion, additional analysis using first order conditions suggests that optimisers should ``support'' the convex set $\C$. In fact, the following example (which builds on Proposition \ref{prop: conv_mon_max}) demonstrates that these optimisers are indeed \num s of $\C$, elaborating on the connection of \num s and support points mentioned in Remark \ref{rem: supp_points}.

Consider a \textsl{utility random field} $U: \Omega \times (0, \infty) \mapsto \Real$, such that $U(\cdot, x) \in \lz$ for all $x \in (0, \infty)$ and $U(\omega, \cdot) : (0, \infty) \mapsto \Real$ is a strictly increasing, concave and continuously differentiable function for all $\omega \in \Omega$. Define the derivative (with respect to the spatial variable) random field $U': \Omega \times (0, \infty) \mapsto \Real_+$ in the obvious way. By means of continuity, the definition of $U$ and $U'$ is extended so that $U(\cdot, 0) \dfn \lim_{x \downarrow 0} U(\cdot, x)$ and $U'(\cdot, 0) \dfn \lim_{x \downarrow 0} U'(\cdot, x)$---note that the latter random variables may take with positive probability the values $- \infty$ and $\infty$, respectively. Assume in the sequel that the \textsl{Inada condition} $\prob \bra{U'(0) = \infty} = 1$ holds, and that $\expec \bra{0 \vee U(\infty)} < \infty$, where $U(\infty) \dfn \lim_{x \to \infty} U(\cdot, x)$ and ``$\vee$'' denotes the maximum operation. Define the functional $\U: \lzp \mapsto [- \infty, \infty)$ via $\U(f) = \expec \bra{U(f)}$, where for $f \in \lzp$ the map $U(f) : \Omega \mapsto [- \infty, \infty)$ is defined via $\pare{U(f)} (\omega) = U(\omega, f(\omega))$ for $\omega \in \Omega$. Clearly, $\U$ is concave and monotone. A combination of $\expec \bra{0 \vee U(\infty)} < \infty$ and Fatou's lemma implies that $\U$ is upper semi-continuous.

Consider a convex, max-closed and bounded $\C \subset \lzp$ with $\C \cap \lzpp \neq \emptyset$. Proposition \ref{prop: conv_mon_max} provides the existence of $g \in \Cmax$ such that $\U(g) = \sup_{f \in \C} \U(f) < \infty$. In fact, because $\U$ is strictly concave, the previous maximiser in unique. In order to avoid unnecessary technical complications, a final mild assumption involving the optimiser $g$ is introduced: we impose that $\U(a g) > - \infty$ holds for all $a \in (0,1)$, which implies that the function $(0, \infty) \ni a \mapsto \U(a g)$ is concave, strictly increasing and $\Real$-valued. Such mapping must have a finite (right-hand-side) derivative; then, a use of the monotone convergence theorem gives that $\expec \bra{U'(ag) g \indic_{\set{g > 0}}}  < \infty$ holds for all $a \in (0, \infty)$. Define the convex set $\C_g \dfn \set{f \in \C \such f \geq a g \text{ for some } a \in (0,1)}$. Note that $g \in \C_g$; furthermore, since for all $f \in \C$ the $\C_g$-valued sequence $\pare{(1 - n^{-1}) f + n^{-1} g}_{\nin}$ converges to $f$, $\C_g$ is dense in $\C$. Fix $f \in \C_g$ and let $a \in (0,1)$ be such that $f \geq a g$. Since $\U(f) \geq \U(a g) > - \infty$, it holds that $\prob[U(f) = - \infty] = 0$, i.e., $U(f) \in \lz$. Similarly, $\U(g) > - \infty$ implies $U(g) \in \lz$. For $\epsilon \in (0, 1)$, define
\[
f_\epsilon \dfn (1 - \epsilon) g + \epsilon f, \text{ and } \Delta(f_\epsilon \such g) \dfn \frac{U(f_\epsilon) - U(g)}{\epsilon} \in \lz.
\]
The optimality of $g$ gives $\expec  \bra{\Delta(f_\epsilon \such g)} \leq 0$, for all $\epsilon \in (0, 1)$. Note that $\Delta(f_\epsilon \such g) \geq 0$ holds on $\set{f_\epsilon \geq g}$; furthermore, for all $\epsilon \in (0, 1)$, $f_\epsilon \geq a g$ implies that $\Delta(f_\epsilon \such g) \geq - U'(f_\epsilon) (g - f) \geq - U'(a g) g \indic_{\set{g > 0}}$ holds on $\set{f_\epsilon < g}$ . Since $\expec  \bra{U'(a g) g \indic_{\set{g > 0}}} < \infty$, and $\liminf_{\epsilon \downarrow 0} \Delta(f_\epsilon \such g) = U'(g) (f - g)$ holds in the $\prob$-a.s. sense, Fatou's lemma implies that $\expec \bra{U'(g) (f - g)} \leq 0$, where in particular $\prob \bra{U'(g) (f - g) = - \infty} = 0$ is implied. By assumption, there exists $h \in \C \cap \lzpp$. It can be assumed without loss of generality that $h \in \C_g$ (otherwise, replace $h$ by $(h + g)/2$); therefore, $\prob \bra{U'(g) (h - g) = - \infty} = 0$ implies that $g \in \lzpp$, which in particular implies that $U'(g) \in \lzp$. The fact that $\expec \bra{U'(g) g} = \expec \bra{U'(g) g \indic_{\set{g > 0}}} < \infty$ holds allows to write $\expec \bra{U'(g) (f - g)} \leq 0$ as $\expec \bra{U'(g) f} \leq \expec \bra{U'(g) g}$ for all $f \in \C_g$. Upon defining the probability measure $\qprob$ via the recipe $\ud \qprob = \pare{ U'(g) g / \expec  \bra{U'(g) g} } \ud \prob$, note that $\qprob \sim \prob$ and $\expecq \bra{f/ g} \leq 1$ holds for all $f \in \C_g$. As $\C_g$ is dense in $\C$, Fatou's lemma implies that $\expecq \bra{f/ g} \leq 1$ holds for all $f \in \C$. Therefore, $g \in \Cnum$.

\begin{rem} \label{rem: Cnum non-empty}
It is straightforward to construct strictly increasing, concave and continuously differentiable deterministic functions $U: (0, \infty) \mapsto \Real$ such that $U'(0) = \infty$ and $U \pare{(0, \infty)}$ is a bounded subset of $\Real$. With $\U$ defined as in the example above, if $\C$ is convex, max-closed and bounded then the unique maximiser $g$ of $\U$ over $\C$ trivially satisfies $\U(a g) > - \infty$ for all $a \in (0,1)$; therefore, $g \in \Cnum$. In particular, we deduce that $\Cnum \neq \emptyset$ holds whenever $\C \subset \lzp$ with $\C \cap \lzpp \neq \emptyset$ is convex, max-closed and bounded.
\end{rem}

\begin{rem}
In the discussion of the above example, under certain assumptions on the utility random field $U$, the set $\C$ and the optimiser $g \in \Cmax$, it is concluded that $g \in \Cnum$. The most restrictive assumption is the boundedness from above of the utility random field, encoded in the requirement $\expec \bra{0 \vee U(\infty)} < \infty$. This assumption is there to ensure that $\U$ is $[-\infty, \infty)$-valued and upper semi-continuous, in order to allow the invocation of Proposition \ref{prop: conv_mon_max} and obtain existence of an optimiser $g \in \Cmax$. However, if existence of an optimiser $g \in \C$ can be obtained with other methods, in which case Lemma \ref{lem: zorn} ensures that it can be additionally assumed that $g \in \Cmax$, the discussion of the above example goes through even without enforcing boundedness conditions on $U$. (The other, milder, assumptions should of course still be satisfied.) There has been a significant body of work in the field of mathematical finance where existence of optimisers for such types of expected utility maximisation problems is established using convex duality methods; for more examples, see \cite{MR2023886} in the case of deterministic $U$ and \cite{KarZit03} for the case where $U$ may actually be a random field.
\end{rem}

\subsection{Maximal points versus \num s} \label{subsec: exa}

Let $\C$ is convex, max-closed and bounded and such that $\C \cap \lzpp \neq \emptyset$. As was discussed in Remark \ref{rem: num_is_max} and Remark \ref{rem: Cnum non-empty}, it holds that $\Cnum \neq \emptyset$ and $\Cnum \subseteq \Cmax \cap \lzpp$. However, the inclusion $\Cnum \subseteq \Cmax \cap \lzpp$ can be strict, as will be shown below by an example, which has also appeared in \cite{MR2929089}. (Note that there are indeed special---but important---cases where $\Cnum \subseteq \Cmax \cap \lzpp$ can be established; for example, see \cite{MR1381678}.)

Consider the probability space $(\Omega, \F, \prob)$, where $\Omega = (0, \infty)$, $\F$ the Borel $\sigma$-field over $(0, \infty)$, and $\prob$ is a probability measure equivalent to Lebesgue measure on $(0, \infty)$. Define $\xi: \Omega \mapsto (0, \infty)$ via $\xi(\omega) = \omega$ for all $\omega \in (0, \infty)$. Furthermore, define $K \dfn \set{(\alpha, \beta) \in \Real^2 \such 0 \leq \beta \leq \sqrt{\alpha} \leq 1}$, and note that $K$ is a convex and compact subset of $\Real^2_+$. Let $\C \dfn \set{1 - \alpha + (\alpha + \beta) \xi \such (\alpha, \beta) \in K}$. Being the image of $K$ via a continuous linear mapping, $\C$ is a convex and compact subset of $\lzp$---therefore, it is closed (in particular, max-closed) and bounded.

Note that $\prob [\xi \leq \epsilon] > 0$ and $\prob [\xi^{-1} \leq \epsilon] > 0$ hold for all $\epsilon \in (0, \infty)$; given this, $\Cmax = \set{1 - \alpha + (\alpha + \sqrt{\alpha}) \xi \such \alpha \in [0, 1]} \subseteq \lzpp$ follows in a rather straightforward way. In particular, it holds that $1 \in \Cmax \cap \lzpp$. However, we claim that $1 \notin \Cnum$. In fact, we shall show that there cannot exist any $\qprob \ll \prob$ such that $\expecq \bra{f} \leq 1$ holds for all $f \in \C$. To wit, if such a probability measure $\qprob$ existed, $\expecq \bra{1 - \alpha + (\alpha + \sqrt{\alpha}) \xi} \leq 1$ for all $\alpha \in [0, 1]$ would follow. Rearranging, $\expecq \bra{\xi} \leq \alpha / \pare{\alpha + \sqrt{\alpha}} = \sqrt{\alpha} / \pare{\sqrt{\alpha} + 1}$ would hold for all $\alpha \in (0, 1]$. This would imply that $\expecq \bra{\xi} = 0$, i.e., $\qprob \bra{\xi > 0} = 0$ which, in view of $\prob[\xi > 0] = 1$, contradicts the fact that $\qprob$ is a probability measure which is absolutely continuous with respect to $\prob$.

In fact, one can say more: in this example, it holds that $\Cnum = \Cmax \setminus \set{1}$. Indeed, fix $\gamma \in (0,1]$ and define $g_\gamma \dfn 1 - \gamma + (\gamma + \sqrt{\gamma}) \xi$; we shall show that $g_\gamma \in \Cnum$. Note that the law of the random variable $1/g_\gamma$ under $\prob$ is equivalent to Lebesgue measure on $\pare{0, (1 - \gamma)^{-1}}$. Therefore, setting $c_\gamma \dfn \pare{1 + 2 \sqrt{\gamma} } \pare{1 + \sqrt{\gamma}}^{-2}$, the strict inequality $c_\gamma < 1 \leq (1 - \gamma)^{-1}$ implies that there exists a probability $\qprob_\gamma \sim \prob$ such that $\expec_{\qprob_\gamma} \bra{ 1 / g_\gamma} = c_\gamma$. The straightforward calculation $\xi / g_\gamma = \gamma^{-1/2} (1+\sqrt{\gamma})^{-1} - \gamma^{-1/2} (1 - \sqrt{\gamma}) (1 / g_\gamma)$ implies $\expecqg[\xi/g_\gamma] = \gamma^{-1/2} (1+\sqrt{\gamma})^{-1} - \gamma^{-1/2} (1 - \sqrt{\gamma}) c_\gamma = 2 \sqrt{\gamma} \pare{1 + \sqrt{\gamma}}^{-2}$. Therefore, it follows that
\[
\expecqg \bra{ \frac{1 - \alpha + (\alpha + \beta) \xi}{g_\gamma} } = \frac{1 + 2 \sqrt{\gamma} - \alpha + 2 \sqrt{\gamma} \beta}{\pare{1 + \sqrt{\gamma}}^{2}} \leq \frac{1 + 2 \sqrt{\gamma} - \alpha + 2 \sqrt{\gamma \alpha}}{\pare{1 + \sqrt{\gamma}}^{2}}, \quad \text{for all } (\alpha, \beta) \in K.
\]
It is easily seen that the latter expression, as a function of $(\alpha, \beta) \in K$, is maximised when $(\alpha, \beta) = \pare{\gamma, \sqrt{\gamma}}$, and that the maximum is $1$. It indeed follows that $g_\gamma \in \Cnum$. 

Before abandoning this example, a final remark is in order. Even though $1 \in \Cmax \setminus \Cnum$, note that the $\Cnum$-valued sequence $\pare{(1 - n^{-1}) + \pare{n^{-1} + n^{-1 /2}} \xi}_{\nin}$ actually converges to $1$. Theorem \ref{thm: bishop-phelps} in the next section will generalise this observation.

\section{Density of Num\'eraires in Maximal Elements} \label{sec: BP}

\subsection{The main result}

What follows is a density result of \num s in maximal elements for convex, max-closed and bounded sets of $\lzp$ that contain at least one strictly positive element.

\begin{thm} \label{thm: bishop-phelps}
Let $\C \subset \lzp$ be convex, max-closed and bounded, and such that $\C \cap \lzpp \neq \emptyset$. Then, $\Cnum$ is dense in $\Cmax$.
\end{thm}

Keeping in mind the discussion in Remark \ref{rem: supp_points}, the statement of Theorem \ref{thm: bishop-phelps} bears resemblance to the celebrated result of Bishop and Phelps \cite[Theorem 7.43, statement 1]{MR2378491}, stating that support points of closed and convex sets in Banach spaces are dense on the boundary of the set. Note, however, that the present setting is by all means non-standard, especially since $\lz$ typically fails to be locally convex. For a convex and bounded $\C \subset \lzp$, there exists a probability $\qprob \sim \prob$, such that $\sup_{f \in \C} \expecq \bra{f} < \infty$; see, for example, \cite[combination of Lemmata 1, 2 and 3 of page 147]{MR2273672}. 
This fact seems to provide hope that one could use the classical version of the Bishop-Phelps theorem by applying $\Lb^1(\qprob)$-$\Lb^\infty$ duality. In fact, under the assumptions of Theorem \ref{thm: bishop-phelps} it is not hard to see that, if $\C \subseteq \Lb^1_+(\qprob)$, then $\Cmax$ is actually contained in the $\Lb^1(\qprob)$-topological boundary of $\C$. However, for a given $g \in \Cmax$, it is not at all clear that the sequence of (usual) support points that approximates $g$ is $\Cnum$-valued. As the previous issue does not appear \emph{a priori} trivial, a bare-hands alternative route is taken in the proof of Theorem \ref{thm: bishop-phelps}, given in \S~\ref{subsec: proof of b-p} below.


\subsection{Proof of Theorem \ref{thm: bishop-phelps}} \label{subsec: proof of b-p}

Assume that $\C$ is convex, max-closed and bounded, and such that $\C \cap \lzpp \neq \emptyset$. Let $g \in \Cmax$; we shall show that there exists a $\Cnum$-valued sequence $(f_n)_{\nin}$ such that $\limn f_n = g$. We first treat the case where $g \in \Cmax \cap \lzpp$; then, the general case will follow through an approximation argument.

\subsubsection{Case where $g \in \Cmax \cap \lzpp$}

In order to obtain the approximating sequence $(f_n)_{\nin}$, we shall use the construction of \S~\ref{subsec: util_max_num}. For fixed $\nin$, define $U_n : \Omega \times (0, \infty) \mapsto (- \infty, 0)$ via
\[
U_n (x) = - \pare{g / x}^{n}, \quad \forall x \in (0, \infty),
\]
where the dependence on $\omega$ (coming from $g \in \C$) is suppressed, as usual. Note that $U_n$ is concave, strictly increasing, continuously differentiable, bounded above by zero, and that the Inada condition $U_n' (0) = \infty$ is satisfied for all $\nin$. Define $\U_n : \lzp \mapsto [- \infty, 0)$ via $\U_n(f) = \expec \bra{U_n(f)}$ for all $f \in \lzp$ and $\nin$. Note that $\U_n(f) > - \infty$ implies $\U_n(a f) > - \infty$ for all $a \in (0,1)$. In view of the general example in \S~\ref{subsec: util_max_num}, for all $\nin$ we infer the existence of $f_n \in \Cnum$ with the property that $\U_n(f_n) = \sup_{f \in \C} \U_n(f)$. It remains to show that $\limn f_n = g$.


Note first that, since $\U_n(g) = \expec \bra{U_n(g)} = - 1$, it follows that $\expec \bra{U_n(f_n)} \geq -1$ for all $\nin$; in other words, $\expec \bra{\pare{g / f_n}^n} \leq 1$ holds for all $\nin$. In view of Markov's inequality, it holds that
\[
\prob \bra{f_n / g < \beta} \leq \beta^{n} \expec \bra{ \pare{g / f_n}^n} \leq  \beta^{n}, \quad \forall \nin \text{ and } \forall \beta \in (0,1).
\]
The Borel-Cantelli lemma implies that for any fixed $\beta \in (0,1)$, $\beta g \leq \liminf_{n \to \infty} f_n$ holds in the $\prob$-a.s. sense. It then follows that $g \leq \liminf_{n \to \infty} f_n$ holds in the $\prob$-a.s. sense.

We proceed in showing that $\limn \prob \bra{f_n / g > 1 + \epsilon} = 0$ holds for all $\epsilon \in (0,\infty)$. Assume on the contrary that there exists $\epsilon > 0$ and a subsequence $(f_{n_k})_{\kin}$ of $\pare{f_n}_{\nin}$ such that $\prob \bra{f_{n_k} / g > 1 + \epsilon} > \epsilon$ holds for all $\kin$. Since $\C$ is convex and bounded, \cite[Lemma A1.1]{MR1304434} gives the existence of a sequence $\pare{h_k}_{\kin}$ of forward convex combinations of $(f_{n_k})_{\kin}$ that converges to some $h \in \lzp$; since $\C$ is convex, it follows that $h \in \oC$. More precisely, write $h_k = \sum_{m=k}^{l_k} \alpha_{k, m} f_{n_m}$ for all $k \in \Natural$, where $l_k \geq k$, $\alpha_{k, m} \geq 0$ for all $k \in \Natural$ and $m \in \set{k, \ldots, l_k}$, as well as $\sum_{m=k}^{l_k} \alpha_{k, m} = 1$. Convexity implies that
\[
\expec \bra{\pare{g / h_k}^{n_k}} \leq \sum_{m=k}^{l_k} \alpha_{k, m} \expec \bra{\pare{g / f_{n_m}}^{n_k}}, \quad \forall \kin.
\]
Jensen's inequality gives $\expec \bra{\pare{g / f_{n_m}}^{n_k}} \leq \pare{ \expec \bra{\pare{g / f_{n_m}}^{n_m}} }^{n_k / n_m} \leq 1$, for all $\Natural \ni k \leq m \in \Natural$. A combination of the previous gives $\expec \bra{ \pare{g / h_k}^{n_k}} \leq 1$, for all $\kin$. As before, this implies that $g \leq \liminf_{k \to \infty} h_k$ holds in the $\prob$-a.s. sense; in particular, $g \leq h$. On the other hand, the fact that $\prob \bra{f_{n_k} / g> 1 + \epsilon} > \epsilon$ holds for all $\kin$, combined with $\liminf_{k \to \infty} \pare{f_{n_k} / g} \geq 1$ holding in the $\prob$-a.s. sense, implies that $\limsup_{k \to \infty} \expec \bra{\exp(- f_{n_k} / g)} \leq (1 - \epsilon) \exp(-1) + \epsilon \exp(-1 - \epsilon)$. Then, convexity and boundedness of the function $(0, \infty) \ni x \mapsto \exp(- x) \in (0, 1)$ implies that
\begin{align*}
\expec \bra{\exp(- h/g)} &= \lim_{k \to \infty} \expec  \bra{\exp(- h_k/g)} \\
&\leq \limsup_{k \to \infty} \pare{ \sum_{m=k}^{l_k} \alpha_{k, m} \expec \bra{\exp(- f_{n_m}/h)} } \\
&\leq  \limsup_{k \to \infty}  \expec  \bra{\exp(- f_{n_k} /h)}  \leq (1 - \epsilon) \exp(- 1) + \epsilon \exp(- 1 - \epsilon) < \exp(-1).
\end{align*}
We obtain that $\prob \bra{h/g > 1} > 0$, which together with $g \leq h$ contradicts the fact that $g \in \Cmax = \oCmax$, the last set-equality coming from Lemma \ref{lem: bdd_max_closed}. Therefore, $\limn \prob \bra{f_n / g> 1 + \epsilon} = 0$ holds for all $\epsilon \in (0,\infty)$; coupled with the fact that $g \leq \liminf_{n \to \infty} f_n$ holds in the $\prob$-a.s. sense that was previously established, we conclude that $\limn f_n = g$.

\subsubsection{Case of arbitrary $g \in \Cmax$}

Let $g \in \Cmax$, and fix some $f \in \C \cap \lzpp$. For all $\nin$, set $h_n \dfn (1 - 1/n) g + (1 / n) f$ and note that $h_n \in \C \cap \lzpp$. Furthermore, by Lemma \ref{lem: zorn} it follows that for each $\nin$ there exists $g_n \in \Cmax$ with $h_n \leq g_n$; of course, $g_n \in \Cmax \cap \lzpp$ holds for all $\nin$. According to what we have already proved, there exists a $\Cnum$-valued sequence $(f_n)_{\nin}$ such that $\limn \pare{g_n - f_n} = 0$ holds. Theorem \ref{thm: bishop-phelps} will be fully established if we can show that $\limn g_n = g$ holds.  Since $\limn h_n = g$, $g \in \Cmax$ and $h_n \leq g_n$ holds for all $\nin$, this fact follows from statement (2) of Proposition \ref{prop: stab_conv_outer_bdry}, which completes the proof.

%
%
%

%

%
%
%
%

\bibliographystyle{alpha}
\bibliography{spsp}
\end{document}